\documentclass[12pt]{article}
\usepackage{graphicx}
\usepackage{amsmath}
\usepackage{authblk}

\usepackage[T5]{fontenc}
\usepackage[utf8]{inputenc}
\usepackage[english]{babel}
\usepackage[margin=1in]{geometry}
\usepackage[pagewise]{lineno}
\usepackage[colorlinks=true, bookmarks=true, pdfstartview=FitH, pagebackref=true, linktocpage=true]{hyperref}

\usepackage{blindtext}
\usepackage{amssymb}
\usepackage{hyperref}
\hypersetup{
    colorlinks=true,
    linkcolor=blue,
    filecolor=magenta,      
    urlcolor=cyan,
}
 
\usepackage{amsthm}
\usepackage{amsmath}
\newtheorem{theorem}{Theorem}[section]

\newtheorem{proposition}{Proposition}[section]
\theoremstyle{remark}

\theoremstyle{definition}

\numberwithin{equation}{section}

\makeatletter
\newcommand{\subjclass}[2][2010]{%
  \let\@oldtitle\@title%
  \gdef\@title{\@oldtitle\footnotetext{#1 \emph{Mathematics Subject Classification.} #2}}%
}
\newcommand{\keywords}[1]{%
  \let\@@oldtitle\@title%
  \gdef\@title{\@@oldtitle\footnotetext{\emph{Keywords.} #1.}}%
}
\makeatother

\def\Z{\mathbb Z}% integers 
\def\R{\mathbf R} %real number
\def\C{\mathbb C} %complex number
\def\T{\mathbb T} %torus

\def\bB{\mathcal B} %bilinear form
\def\Csf{\mathsf c}
\def\sN{\mathsf N}
\def\sM{\mathsf M}

\def\d{\partial}

\newif\ifprint
\ifprint
\else
\fi

\begin{document}

\title{A nonlinear instability result to the Navier-Stokes equations with Navier slip boundary conditions}
\author{Tien-Tai Nguyen \\ Email: \href{mailto: nttai.hus@vnu.edu.vn}{nttai.hus@vnu.edu.vn}} 
\affil{VNU University of Science,\\
 334 Nguyen Trai street, Hanoi, Vietnam}


\maketitle

\begin{abstract}
In this paper, we investigate the  instability  of the trivial steady states to the incompressible viscous fluid with Navier-slip boundary conditions. For the linear instability, the existence of \textit{finitely many} normal mode solutions to the linearized equations is shown via the operator method of Lafitte and Nguyen (2022). Hence, we prove the nonlinear instability by adapting the framework of Desjardins and Grenier (2003) studying  some classes of viscous boundary layers to obtain two separated solutions at escaping time. Our work performs a different approach from that of Ding, Li and Xin (2018).
\end{abstract}

%\tableofcontents

\section{Introduction}

\subsection{Formulation of the problem}

Let $\T$ be the usual 1D torus and let us consider a  horizontal slab domain $\Omega = 2\pi L \T \times (-1,1)$, with the length of periodicity $L>0$. In this paper, we consider an incompressible fluid, governed by the Navier-Stokes equations
\begin{equation}\label{EqNS}
\begin{cases}
\partial_t u + u \cdot \nabla u + \nabla p - \mu \Delta u = 0, & \text{in } \Omega, \ t \ge 0, \\
\nabla \cdot u = 0, & \text{in } \Omega, \ t \ge 0,
\end{cases}
\end{equation}
where $u(x_1,x_2,t) = (u_1(x_1,x_2,t),u_2(x_1,x_2,t))$ and $p(x_1,x_2,t)$ are the velocity and pressure of the flow respectively. 

The Navier-Stokes equations are frequently studied with Dirichlet boundary condition   (see e.g. \cite{Tem84}), i.e. the fluid does not slip along the boundary, formulated by Stokes in 1845. There are also other types of boundary conditions in physical phenomena.  For instance, the Navier-slip boundary conditions is proposed by Navier \cite{Nav23}  since 1827,  that allows the fluid to slip and that will be considered in this paper. Another type of boundary conditions for \eqref{EqNS} is the diffusive-free boundary conditions, recently studied  by  Lin and Kerswell \cite{LK24} and Dormy and Gerard-Varet \cite{DGV25}. 
%However, this condition is not always realistic and leads to induce a strong boundary layer in general. For example, hurricanes and tornadoes, do slip along the ground, lose energy as they slip and do not penetrate the ground (see \cite{AdO07}). Other examples about the slip of the fluid on the boundary occur when moderate pressure is involved such as in high altitude aerodynamics (see \cite{HWLS16}), or in immiscible two phase flows, the moving contact line is not compatible with no-slip boundary condition, see \cite{AS11}. 

Let $\Sigma_\pm = 2\pi L\T \times \{x_2=\pm 1\}$,  the Navier-slip boundary conditions  are given on $\Sigma_\pm$ as follows 
\begin{equation}\label{Boundary}
\begin{split}
u \cdot n &= 0 \quad \text{on } \Sigma_+\cup \Sigma_-,\\
\big[ (-pI + \mu (\nabla u + (\nabla u)^T)) \cdot n \big] \cdot {\tau} &= \xi(x) u \cdot {\tau}, \quad   \quad \text{on } \Sigma_+ \cup \Sigma_-,\\
\end{split}
\end{equation}
where the superscript $T$ means matrix transposition, $I$ is the $2 \times 2$ identity matrix,  $n$ is the outward unit normal vector and ${\tau}$ is the corresponding tangent vector of the boundary. In \eqref{Boundary}, for simplicity, let  $\xi(x)$ be a scalar function describing the slip effect on the boundary, only taking  constant values $\xi_\pm$ on $\Sigma_\pm$, respectively.  We refer to other papers \cite{Kel06, GK12, KT24}, where the authors consider variable $\xi$.

For Navier boundary value problems, the regime $\xi_\pm\leq 0$, which models the dissipative interaction between the fluid and the boundary, has received considerable attention in both the mathematical and physical literature over the past several decades. One of the first study on the well-posedness of Navier–Stokes equations with Navier boundary conditions was due to Solonnikov and Ščadilov \cite{SS73} for the stationary equations,  see also some papers of Amrouche and his collaborators \cite{AS11, AR14, A++21}. For analytical study on the time--dependent Navier–Stokes equations in a bounded domain or on the half-space, we refer to Clopeau et al. \cite{CMR98}, or to da Veiga \cite{dV05}, showing the existence of the regular solution. Let us mention also other problems related to the incompressible Navier–Stokes equations with Navier-slip  boundary conditions, the control problem studied by Coron \cite{Cor96} and the inviscid limit problem, studied by  Iftimie and Sueur \cite{IS11}.  For  numerical study, we refer to \cite{APV98, Ban01, BJ67, JM00, JM01, Joh02, QWS03, Ser59}.

%As Serrin pointed out in 1959 (see \cite[Page 240]{Ser59}), there is no fundamental reason to impose a sign restriction on the parameter $\xi_\pm$. In contrast to the classical regime $\xi_\pm\leq 0$, the case $\xi_\pm >0$ has been studied much less extensively. Nevertheless, positive values of $\xi_\pm$  arise naturally in a variety of physically relevant settings. For instance, the effective slip length is always positive for flat hybrid gas--liquid interfaces \cite{HWLS16}. Positive slip conditions are also commonly used to model flows over hydrophobic and superhydrophobic surfaces, where reduced fluid--solid friction can produce large apparent slip lengths \cite{FBV09,Rot10}. Moreover, positive slip coefficients occur in flows along permeable boundaries, where the classical Beavers--Joseph condition describes the interfacial slip between a free fluid and a porous medium \cite{BJ67,Saf71}. Such boundary conditions have been employed in various applications, including Couette flow and heat transfer in the presence of naturally permeable walls \cite{CS93}.

Thanks to Serrin \cite[p.~240]{Ser59}, there is no fundamental reason to impose a sign restriction on the parameters $\xi_\pm$ . Although the classical regime $\xi_\pm\leq 0$  has been studied extensively, the case $\xi_\pm>0$  remains relatively less explored. From the physical point of view, however, positive slip coefficients arise in a variety of important applications. For instance, flat hybrid gas--liquid interfaces are characterized by a positive effective slip length \cite{HWLS16}. Similar effective slip conditions are widely employed to model flows over hydrophobic and superhydrophobic surfaces, where interfacial slip results in significant drag reduction \cite{FBV09,Rot10}. Positive slip also appears in flows over permeable media, where the Beavers--Joseph condition describes the interfacial slip between a free fluid and a porous material \cite{BJ67,Saf71}. These models have been successfully applied to numerous engineering problems, including Couette flows and heat transfer in channels with naturally permeable boundaries \cite{CS93}.

\subsection{The goal of this paper}

Our main interest here is to study the nonlinear  instability of the steady state solution $(0, p_s)$ ($p_s$ is a constant) to this boundary value problem \eqref{EqNS}-\eqref{Boundary}, revisiting the previous result of Ding, Li and Xin \cite{DLX18} in a slab domain.   In that paper, the authors obtain  a threshold of viscosity  $\mu_c$ depending on positive slip coefficients $\xi_\pm$, that separates the regimes of  nonlinear stability and nonlinear instability. We also refer to another paper \cite{LPZ22} solving the case of bounded domain.  In this paper, we aim at showing the nonlinear instability of the state $(0, p_s)$ in the subcritical regime  $\mu<\mu_c$ by an alternative method, inspired by Lafitte and the author \cite{LN20} and by  Desjardins and Grenier \cite{DG03}

To prove the linear instability, we study the spectral analysis following normal mode ansatz of Chandrasekhar \cite{Cha61}. Precisely, for any horizontal spatial frequency $k\in L^{-1}\Z\setminus\{0\}$, we define the $k$-subcritical regime of the viscosity coefficient $\mu < \mu_c(k,\Xi)$ (see $\mu_c(k,\Xi)$ in Proposition \ref{PropMuC}). Thus we prove that there exists a finite sequence of normal mode solutions to the linearized equations thanks to the operator method of Lafitte and the author.  This is stated in Theorem \ref{ThmSpectral}.

The second part is to prove the nonlinear instability in the subcritical regime 
\begin{equation}\label{mu-c}
\mu< \mu_c(\Xi)=  \sup_{k\in L^{-1}\Z\setminus\{0\}} \mu_c(k,\Xi). 
\end{equation}
Following Desjardins and Grenier \cite{DG03}, where they studied the nonlinear instability of some classes of boundary layers to the incompressible rotating fluid, we prove the nonlinear instability in the sense of \cite[Theorem 2.3]{DG03}.  Thanks to the existence of finitely many normal mode solutions to the linearized equations \eqref{EqLinearized}, we intend to construct two separated solutions at the escaping time. The statement will be shown in Theorem \ref{ThmNonlinear}.

\section{Main results}

Denote the perturbation by 
\[
u=u-0, \quad q=\pi-\pi_s.
\]
Hence, $(u,q)$ satisfies the nonlinear equations
\begin{equation}\label{EqMain}
\begin{cases}
\d_t u+ u\cdot\nabla u +\nabla q -\mu\Delta u =0,\\
\text{div}u=0.
\end{cases}
\end{equation}
In our setting, $n = (0,1)$ on $\Sigma_+$ and $n = (0,-1)$ on $\Sigma_-$, while $\tau = (1,0)$ on both $\Sigma_\pm$.  Hence, the boundary conditions \eqref{Boundary} in the slab domain can be written under the form
\begin{equation}\label{NavierBound}
\begin{cases}
u_2 =0 &\quad\text{on } \Sigma_+ \cup \Sigma_-,\\
\mu \d_2 u_1= \xi_+ u_1 &\quad\text{on }\Sigma_+,\\
\mu \d_2 u_1=- \xi_- u_1 &\quad\text{on }\Sigma_-.
\end{cases}
\end{equation}
From now on, we move to study the instability of trivial state to \eqref{EqMain}-\eqref{NavierBound}.

Linearizing \eqref{EqMain}--\eqref{NavierBound} around the trivival steady state yields the linearized equations
\begin{equation}\label{EqLinearized}
\begin{cases}
\d_t u +\nabla q -\mu\Delta u=0 &\quad\text{in }\Omega,\\
\text{div}u=0 &\quad\text{in }\Omega,\\
u_2 =0 &\quad\text{on } \Sigma_+ \cup \Sigma_-,\\
\mu \d_2 u_1= \xi_+ u_1 &\quad\text{on }\Sigma_+,\\
\mu\d_2 u_1=- \xi_- u_1 &\quad\text{on }\Sigma_-.
\end{cases}
\end{equation}
We will solve problem \eqref{EqLinearized} by the standard normal mode analysis, see \cite{Cha61}. That means, we look for $u$ and $q$ of the form
\begin{equation}\label{NormalMode}
\begin{cases}
u_1(t,x_1,x_2)= e^{\lambda t}\sin (kx_1)\psi(x_2),\\
u_2(t,x_1,x_2)= e^{\lambda t} \cos(kx_1) \phi(x_2),\\
q(t,x_1,x_2)= e^{\lambda t} \cos(kx_1)\pi(x_2),
\end{cases}
\end{equation}
where $k\in L^{-1} \Z \setminus \{0\}$ is called the wavenumber and $\lambda=\lambda(k) \in \C$ is called a characteristic value of the linearized equations (after \cite{Cha61}). Since we are interested in the linear instability in this section, we look for $\lambda$ with positive real part. Let $'=\frac{d}{dx_2}$, substituting \eqref{NormalMode} into \eqref{EqLinearized}, we obtain the following system of ODEs
\begin{equation}\label{SystemMode}
\begin{cases}
\lambda\psi - k\pi - \mu (k^2 \psi-\psi'')=0,\\
\lambda\phi +\pi' -\mu (k^2\phi-\phi'')=0,\\
k\psi+\phi'=0
\end{cases}
\end{equation}
with boundary conditions
\begin{equation}\label{BoundMode}
\phi(\pm 1)=0, \quad \mu \psi'(1)=\xi_+\psi(1), \quad \mu\psi'(-1)=-\xi_- \psi(-1).
\end{equation}

Eliminating $\pi$ from $\eqref{SystemMode}_1$ and $\psi$ from $\eqref{SystemMode}_3$ gives us a fourth order ODE for $\phi$, that is
\begin{equation}\label{4thOrderEqPhi}
\lambda(k^2\phi-\phi'')+ \mu(\phi^{(4)}-2k^2\phi''+k^4\phi)=0,
\end{equation}
with the boundary conditions
\begin{equation}\label{BoundOde}
\phi(\pm 1)=0, \quad \mu \phi''(1)=\xi_+ \phi'(1), \quad \mu \phi''(-1)=-\xi_- \phi'(-1).
\end{equation}
Hence,  solving the system \eqref{SystemMode} with boundary conditions \eqref{BoundMode} is reduced  to study the ODE \eqref{4thOrderEqPhi} with boundary conditions \eqref{BoundOde}. 

Let $k$ be fixed, we will look for a solution $\phi\in H^4((-1,1))$ to \eqref{4thOrderEqPhi}--\eqref{BoundOde} with positive $\lambda$.
%Ding-Li-Xin \cite{DLX18} exploited the variational structure of \eqref{4thOrderEqPhi}--\eqref{BoundOde} to obtain the smallest characteristic value $\lambda$. 
In this paper, we utilize an operator approach initiated by Lafitte and the author \cite{LN20} to prove the existence of \textit{finitely many} characteristic values $\lambda$ to the linearized equations \eqref{EqLinearized}. The line of investigation is similar to that one of viscous Rayleigh-Taylor instability to incompressible fluid with Navier-slip boundary conditions \cite{Tai22}. That means, we will place ourselves in the $k$-subcritical regime of the viscosity coefficient  $\mu< \mu_c(k,\Xi)$ with $\Xi=(\xi_-,\xi_+)$

To state the linear results, we recall the properties of  $\mu_c(k,\Xi)$ and $\mu_c(\Xi)$ (see \eqref{mu-c}) from \cite[Proposition 3.1]{Tai22}.
\begin{proposition}\label{PropMuC}

Let $k\in \R_+$, instead of $L^{-1}\Z\setminus\{0\}$
and  let 
\[
\tilde H^s((-1,1))= \{\phi \in H^s((-1,1)), \phi(\pm 1)=0\}
\]
with $s\geq1$. The following results hold for any $k>0$.
\begin{enumerate}
\item  We have 
\begin{equation}\label{MuC_kMax}
\begin{split}
\mu_c(k,\Xi) &= \max\limits_{\phi\in \tilde H^2((-1,1))} \frac{\xi_-(\phi'(-1))^2+\xi_+(\phi'(1))^2}{\int_{-1}^1 ((\phi'')^2+2k^2(\phi')^2+k^4\phi^2) dx_2}\\
&= \frac1{4k \sinh^2(2k) } 
\left( \begin{split} 
&( \sinh(2k)\cosh(2k) -2k)(\xi_+ +\xi_-) \\
&+ \left(\begin{split} 
&( \sinh(2k) -2k\cosh(2k))^2 (\xi_++ \xi_-)^2  \\
& + \sinh^2(2k)(\sinh^2(2k)-4k^2) (\xi_+ -\xi_-)^2
\end{split}\right)^{\frac12}
\end{split}\right).
\end{split}
\end{equation}

\item $\mu_c(k,\Xi)$ is a decreasing function in $k\in \R_+$ and satisfies
\begin{equation}\label{Mu_Climit}
\lim_{k\to +\infty} \mu_c(k,\Xi)= 0, \quad \lim_{k\to 0} \mu_c(k,\Xi) = \mu_c(\Xi).
\end{equation}

\item We have
 \begin{equation}\label{EqMuC}
\begin{split}
\mu_c(\Xi) &=\max\limits_{\phi\in \tilde H^2((-1,1))} \frac{\xi_-(\phi'(-1))^2+\xi_+(\phi'(1))^2}{\int_{-1}^1 (\phi'')^2 dx_2}=  \frac13 \Big(\xi_++ \xi_-+ \sqrt{\xi_+^2-\xi_+\xi_- +\xi_-^2}\Big).
\end{split}
\end{equation}
\end{enumerate}
\end{proposition}

Hence, in the subscritical regime $\mu<\mu_c(\Xi)$, we show the following theorems. 
\begin{theorem}\label{ThmSpectral}
 Let $k$ be fixed and $\mu \in (0,\mu_c(k,\Xi))$. There exists finitely many characteristic values 
\[
\lambda_1>\lambda_2>\dots>\lambda_\sM>0,
\]
such that for each $1\leq j\leq \sM <\infty$, there exists a nontrivial function $\phi_j\in H^4((-1,1))$ satisfying \eqref{4thOrderEqPhi}--\eqref{BoundOde} as $\lambda=\lambda_j$. 
\end{theorem}

%\begin{theorem}\label{ThmLinear}
%Let $\mu \in (0,\mu_c(\Xi))$, there exist infinitely normal mode solutions to the linearized equations \eqref{EqLinearized} of the form \eqref{NormalMode} for some wavenumber $k$. 
%\end{theorem}

Once the linear instability is proven, we move to prove the nonlinear instability in the spirit of Desjardins-Grenier's framework \cite{DG03}. The authors in \cite{DLX18} follow the approach of Guo-Strauss \cite{GS95} to obtain an exponentially unstable solution  under $L^2$-norm. Instead of that, we will formulate a linear combination of finitely many normal mode solutions, found in Theorem \ref{ThmSpectral}, to approximate the nonlinear equations \eqref{EqMain}. That helps us to obtain two solutions that diverge at exponentially escaping time.

\begin{theorem}\label{ThmNonlinear}
Assume that $\mu \in(0,\mu_c(\Xi))$. The trivial state  of \eqref{EqMain}--\eqref{NavierBound} is nonlinearly unstable in the following sense: there exist two positive constants $\delta_0$ and $\varepsilon_0$ sufficiently small such that for $\delta \in(0,\delta_0)$, Eq.  \eqref{EqMain}--\eqref{NavierBound}  admits  two solutions $u^{1,\delta}$ and $u^{2, \delta}$ satisfying
\[
\|u^{1,\delta}(0) \|_{H^2} + \|u^{2, \delta}(0)\|_{H^2}  \le  \delta,
\]
and
\[
\|u^{1,\delta}(T^\delta) - u^{2,\delta}(T^\delta)\|_{L^2} \ge m_0\varepsilon_0 > 0
\]
for some positive time $T^\delta=O(|\log \delta|)$, where $m_0 > 0$ is fixed and independent of $\delta$.
\end{theorem}

\section{The linear instability}

\subsection{Auxiliary operators}

In this section, we  study the ODE \eqref{4thOrderEqPhi}-\eqref{BoundOde} for arbitrary $\mu>0$ and for fixed $k$.  Let $\nu> 0$, we rewrite it as
\begin{equation}\label{eq-nu}
\begin{cases}
(-\lambda+\nu)( -\phi'' +k^2\phi) = \mu(\phi^{(4)}-2k^2\phi''+k^4 \phi) + \nu (-\phi''+k^2\phi),\\
\phi(\pm 1)=0, \\
\mu\phi''(1)=\xi_+ \phi'(1), \quad \mu\phi''(-1)=-\xi_- \phi'(-1).
\end{cases}
\end{equation}
 Of importance  is to construct a continuous and coercive bilinear  form $\bB_\nu$ for $\nu\gg 1$ on the functional space $\tilde H^2((-1,1))$, so that  the finding of  a solution $\phi\in H^4((-1,1))$ of Eq.  \eqref{eq-nu}  on $(-1,1)$ is equivalent to finding a weak solution $\phi \in \tilde H^2((-1,1))$ to  the variational problem 
\begin{equation}\label{EqVaritational_Navier} 
 \bB_\nu(\phi, \theta)=(-\lambda+ \nu) \int_{-1}^1   (\vartheta'\varrho'+ k^2\vartheta\varrho) dx_2 \quad\text{for all } \theta \in \tilde H^2((-1,1)),
\end{equation}
and thus improving the regularity of that weak solution $\phi$. 

\begin{proposition}\label{PropOpe_P}
Suppose that $\nu \gg 1$, the followings hold. 
\begin{enumerate}
\item[(i)] The bilinear form 
\begin{equation}\label{BilinearForm}
\begin{split}
\bB_\nu (\vartheta, \varrho)&=   \mu\int_{-1}^1 (\vartheta''\varrho''+ 2k^2\vartheta'\varrho'+ k^4\vartheta\varrho)dx_2 - (\xi_+ \vartheta'(1)\varrho'(1)+ \xi_- \vartheta'(-1)\varrho'(-1)) \\
&\qquad+ \nu  \int_{-1}^1   (\vartheta'\varrho'+ k^2\vartheta\varrho) dx_2
\end{split}
\end{equation}
is continuous and coercive on $\tilde H^2((-1,1))$.
\item[(ii)] There exists a unique operator $Y_\nu$ such that $\bB_\nu(\vartheta, \varrho)=\langle Y_\nu\vartheta, \varrho\rangle$ for all $\varrho \in \tilde H^2((-1,1))$. In a weak sense, we have that
\[
Y_\nu \phi= \nu (k^2\phi-\phi'')+ \mu (\phi^{(4)}-2k^2\phi''+k^4\phi).
\]

\item[(iii)] Let $f\in L^2((-1,1))$ be given, there exists a unique solution $\phi \in H^4((-1,1))$ satisfying the boundary conditions \eqref{BoundOde}  such that $Y_\nu \phi=f$.
\end{enumerate}
\end{proposition}
\begin{proof}

We prove the first item by using the following $\varepsilon$-trace inequality. For every $f\in H^1(-1,1)$ and every $\varepsilon>0$, there exists a constant $C_\varepsilon>0$ such that
\begin{equation}\label{trace}
|f(1)|^2+|f(-1)|^2    \leq\varepsilon \|f'\|_{L^2(-1,1)}^2+C_\varepsilon\|f\|_{L^2(-1,1)}^2.
\end{equation}
Indeed, for any $y\in(-1,1)$, we have
\[
  |f(1)|^2-|f(y)|^2    =2\int_y^1 f(x)f'(x) dx,
\]
it yields
\[
 |f(1)|^2   \leq |f(y)|^2+2\int_{-1}^1 |f(x)f'(x)|\,dx.
\]
Integrating with respect to $y\in(-1,1)$ and using the Cauchy--Schwarz inequality, we obtain
\[
 |f(1)|^2    \leq \frac12 \|f\|_{L^2(-1,1)}^2+\|f\|_{L^2(-1,1)}   \|f'\|_{L^2(-1,1)}.
\]
By Young's inequality, for every $\varepsilon>0$,
\[
 |f(1)|^2    \leq  \frac\varepsilon{2} \|f'\|_{L^2}^2 +C_\varepsilon  \|f\|_{L^2}^2.
\]
The same argument near $x=-1$ yields
\[
    |f(-1)|^2 \leq  \frac\varepsilon{2} \|f'\|_{L^2}^2 +C_\varepsilon  \|f\|_{L^2}^2.
\]
Adding these two inequalities, we obtain \eqref{trace}.

Using \eqref{trace}, we have
\[\begin{split}
 \xi_+ (\vartheta'(1))^2+ \xi_-( \vartheta'(-1))^2 \leq\frac12 \mu \|\vartheta''\|_{L^2}^2+ C_\mu \|\vartheta' \|_{L^2}^2,
\end{split}\]
for some positive $C_\mu$. Hence, 
\[\begin{split}
B_\nu(\vartheta,\vartheta) &= \mu \|\vartheta''\|_{L^2}^2 + (2k^2\mu +\nu) \|\vartheta'\|_{L^2}^2 + (k^4\mu+k^2\nu) \|\vartheta\|_{L^2}^2 -\xi_+ (\vartheta'(1))^2+ \xi_-( \vartheta'(-1))^2\\
&\geq \frac12 \mu \|\vartheta''\|_{L^2}^2 + (2k^2\mu +\nu - C_\mu) \|\vartheta'\|_{L^2}^2 + (k^4\mu+k^2\nu) \|\vartheta\|_{L^2}^2.
\end{split}\]
For $\nu >C_\mu$, we obtain the coerciveness of $B_\nu$. 

The proof of other parts is straightforward thanks to Riesz's representation theorem and a bootstrap argument, so we omit the details.
\end{proof}

Thanks to Proposition \ref{PropOpe_P}, we obtain the following proposition. 

\begin{proposition}\label{PropOpe_S}
\begin{enumerate}
\item The operator $Q\phi = -\phi''+k^2\phi$ from $\tilde H^2((-1,1))$ to $L^2((-1,1))$ is symmetric and positive. 

\item The operator $S_\nu:= Q^{1/2} Y_\nu^{-1}Q^{1/2}$ is compact and self-adjoint from $\tilde H^1((-1,1))$ to itself.
\end{enumerate}
\end{proposition}
\begin{proof}
The proof of Part 1 is obvious. Let us focus on Part 2.

From Part 1, the operator $Q$ has an orthogonal basis of $\tilde H^1$ with  the eigenvalues $\alpha_n\in \R_+$. Hence, the operator $Q^{1/2}$ is a symmetric operator with eigenvalues $\sqrt{\alpha_n}$ and $Q^{1/2}$ is a closed operator because of Sobolev embedding $H^2 \hookrightarrow \tilde H^1\hookrightarrow L^2$. The von Neumann theory suggests that $Q^{1/2}$ is a self-adjoint operator. 

Proposition \ref{PropOpe_P} helps us to define the inverse operator $Y_\nu^{-1}$ of $Y_\nu$, from $L^2((-1,1))$ to a subspace $H^4((-1,1))$ requiring all elements satisfy \eqref{BoundOde}. As a result, we deduce that  $S_\nu$ sends $\tilde H^1((-1,1))$ to $\tilde H^3((-1,1))$. Composing $S_\nu$ with the continuous injection $H^p\hookrightarrow H^q$ for $p>q\geq 0$, we obtain the  compactness and self-adjointness of $S_\nu$. The proof of Proposition \ref{PropOpe_S} is complete.
\end{proof}

\subsection{Normal mode solutions}
 As a result of the spectral theory of compact and self-adjoint operators, the point spectrum of $S_\nu$ is discrete, i.e. is a  sequence $\{ \gamma_n(\nu)\}_{n\geqslant 1}$ of   eigenvalues of $S_\nu$, tending to 0 as $n\to\infty$, associated with normalized orthogonal eigenfunctions $\{\varpi_n\}_{n\geqslant 1}$ in $\tilde H^1((-1,1))$.  That means 
\[
S_\nu\varpi_n= \gamma_n(\nu) \varpi_n .
\]
So that  $\phi_n = Y_\nu^{-1}(Q^{1/2} \varpi_n)$ belongs to  $H^4((-1,1))$ and satisfies \eqref{BoundOde}.  One  thus has 
\begin{equation}\label{EqRf_n}
\gamma_n(\nu) Y_\nu\phi_n = Q \phi_n,
\end{equation}
and $\phi_n$ satisfies \eqref{NavierBound}. Eq. \eqref{EqRf_n} also tells us that $\gamma_n(\nu) >0$ for all $n$. Indeed, we obtain 
\[
\gamma_n(\nu) \bB_\nu(\phi_n,\phi_n)= \gamma_n(\nu)\int_{-1}^1 (Y_\nu\phi_n)  \phi_n dx_2= \int_{-1}^1 [ (\phi_n')^2 +k^2\phi_n^2]dx_2>0.
\]
Hence,  by reordering, we have that $\{\gamma_n(\nu)\}_{n\geq 1}$ is a positive sequence decreasing towards 0 as $n\to \infty$.

We have the following variational formulation of the largest eigenvalue $\gamma_1$.
\begin{proposition}\label{PropLambda1}
There holds 
\begin{equation}
\nu -\frac1{\gamma_1(\nu)}=  \max_{\phi\in \tilde H^2((-1,1))} \frac{\xi_-(\phi'(-1))^2+\xi_+(\phi'(1))^2 - \mu \int_{-1}^1 ((\phi'')^2+2k^2(\phi')^2+k^4\phi^2) dx_2}{\int_{-1}^1 ((\phi')^2+k^2\phi^2) dx_2}\end{equation}
\end{proposition}
\begin{proof}
Set 
\[
\beta = \max_{\phi\in \tilde H^2((-1,1))} \frac{\xi_-(\phi'(-1))^2+\xi_+(\phi'(1))^2 - \mu \int_{-1}^1 ((\phi'')^2+2k^2(\phi')^2+k^4\phi^2) dx_2}{\int_{-1}^1 ((\phi')^2+k^2\phi^2) dx_2},
\]
we prove that $\gamma_1\leq \beta$. Let us consider the Lagrangian functional 
\[\begin{split}
\mathcal L(\phi, \beta) &=  \xi_-(\phi'(-1))^2+\xi_+(\phi'(1))^2 - \mu \int_{-1}^1 ((\phi'')^2+2k^2(\phi')^2+k^4\phi^2) dx_2  \\
&\qquad- \beta\Big( \int_{-1}^1 ((\phi')^2+k^2\phi^2) dx_2-1\Big).
\end{split}\]
Thanks to the Lagrange multiplier theorem, the extrema  of the quotient 
\[
 \frac{\xi_-(\phi'(-1))^2+\xi_+(\phi'(1))^2 - \mu \int_{-1}^1 ((\phi'')^2+2k^2(\phi')^2+k^4\phi^2) dx_2}{\int_{-1}^1 ((\phi')^2+k^2\phi^2) dx_2}
\]
are necessarily the stationary points $( \beta ,\phi_\star)$ of $\mathcal L$, which satisfy 
\begin{equation}\label{PsiConstraintAt1}
\int_{-1}^1 ((\phi_\star')^2+k^2\phi_\star^2) dx_2 =1
\end{equation}
and
\begin{equation}\label{EqConstraintLb}
\begin{split}
&\xi_- \phi_\star'(-1)\theta'(-1)+\xi_+ \phi_\star'\theta'(1) - \mu \int_{-1}^1 ( \phi_\star'' \theta''+2k^2 \phi_\star' \theta'+k^4\phi_\star\theta_\star) dx_2 \\
&\qquad = \beta \int_{-1}^1 (\phi_\star'\theta' +k^2\phi_\star\theta) dx_2,
\end{split}
\end{equation}
for all $\theta \in \tilde H^2((-1,1))$. Restricting $\theta \in C_0^\infty((-1,1))$, one deduces from \eqref{EqConstraintLb} that $\phi_\star$ has to satisfy  
\[
Y_\nu \phi_\star = (\nu- \beta) Q\phi_\star
\] 
in a weak sense. We further get that $ \phi_\star\in H^4((-1,1))$ and satisfies \eqref{PsiConstraintAt1} and the boundary conditions \eqref{NavierBound} after a bootstrap argument.  Hence,  $\frac1{\nu-\beta}$ is an eigenvalue of the compact and self-adjoint operator $ S_\nu$ from $\tilde H^1((-1,1))$ to itself, with $ Y_\nu^{-1} Q^{1/2} \phi_\star$ being an associated eigenfunction. That implies $\frac1{\gamma_1(\nu)} \leq \nu-\beta$.

Now, we prove the reverse inequality $\frac1{\gamma_1(\nu)} \geq \nu-\beta$. Since $S_\nu^{-1}$ is a self-adjoint and positive operator, we have that
\[
\gamma_1= \max_{\omega\in \tilde H^2((-1,1))} \frac{\langle S_\nu^{-1} \omega, \omega\rangle}{\|\omega\|_{\tilde H^2((-1,1))}}.
\]
Hence, let $\widetilde\omega \in \tilde H^2((-1,1))$ be an extremal function, there exists  $\widetilde \phi =Y_\nu^{-1} Q_\nu^{1/2} \widetilde \omega \in H^4((-1,1))$ and we have $\langle Y_\nu \widetilde\phi, \widetilde\phi\rangle = \langle S_\nu \widetilde\omega, \widetilde\omega\rangle$.
It yields
\[
\gamma_1 \langle Y_\nu \widetilde \phi,\widetilde \phi\rangle   = \frac{\langle Y_\nu^{-1} \widetilde \omega, \widetilde\omega\rangle^2}{\|\widetilde \omega\|_{\tilde H^2((-1,1))}} = \|S_\nu^{-1} \widetilde \omega\|_{\widetilde H^2((-1,1))}^2= \langle Q \widetilde\phi, \widetilde\phi\rangle.
\]
Hence,  we have $\frac1{\gamma_1} \langle Q \widetilde\phi, \widetilde\phi\rangle =\langle Y_\nu \widetilde \phi,\widetilde \phi\rangle$. As a result $\frac1{\gamma_1(\nu)} \geq \nu- \beta$. The reverse inequality helps us to complete the proof of Proposition \ref{PropLambda1}.
\end{proof}

Now, we are able to prove Theorem \ref{ThmSpectral}.
\begin{proof}
From \eqref{EqRf_n}, we have that $\phi_n$ is a solution to 
\begin{equation}\label{eq-phi-n}
\begin{cases}
(\nu -\frac1{\gamma_n(\nu)}) ( -\phi_n'' +k^2\phi_n) + \mu(\phi_n^{(4)}-2k^2\phi_n''+k^4 \phi_n)=0,\\
\phi_n(\pm 1)=0, \\
\mu\phi_n''(1)=\xi_+ \phi_n'(1), \quad \mu\phi_n''(-1)=-\xi_- \phi'(-1).
\end{cases}
\end{equation}
Set 
\[
\lambda_n = \nu -\frac1{\gamma_n(\nu)},
\] 
clearly $\lambda_n$ is a decreasing sequence tending to $-\infty$ as $n\to \infty$.  Thanks to Proposition \ref{PropLambda1},  we know that $\lambda_1>0$ in the subcritical regime of viscosity, i.e for $\mu\in \mu_c(k,\Xi)$.
We thus divide the sequence $\lambda_n$ into two parts
\[
\lambda_1 > \lambda_2 > \dots> \lambda_\sM >0 >\lambda_{\sM+1}> \dots.
\]
The number $\sM$ is finite since $\gamma_n$ decreases towards 0 as $n\to \infty$. Plugging into \eqref{eq-phi-n}, there exists finitely many characteristic values $\lambda_n$ $(1\leq n\leq \sM)$ such that   Eq. \eqref{4thOrderEqPhi}- \eqref{BoundOde} has a solution $\phi_n$ as $\lambda=\lambda_n>0$.
\end{proof}

%We now solve the linearized equations \eqref{EqLinearized}, proving Theorem \ref{ThmLinear}.

\section{The nonlinear instability}

In this section, the constant $C$ is a generic constant depending physical parameters. 

\subsection{Linear combination of normal modes}

For any $\mu\in(0,\mu_c(\Xi))$, it follows from Proposition \ref{PropMuC}(2) that there exists a critical wavenumber $k_c\in L^{-1}\Z\setminus\{0\}$ such that 
\[\begin{split}
\mu < \mu_c(k,\Xi) < \mu_c(\Xi) \quad\text{for }  |k|<|k_c|, \quad\text{and } \mu >\mu_c(k,\Xi)  \quad\text{for  }  |k|>|k_c|, 
\end{split}\]
As a consequence of Proposition \ref{PropLambda1}, let 
\begin{equation}\label{GrowthRate}
\Lambda = \max_{k \in L^{-1}\Z\setminus\{0\}, |k|< |k_c|} \lambda_1(k)<+\infty.
\end{equation}

We now fix a wavenumber $k_0\in L^{-1}\Z\setminus\{0\}$ with $|k_0|<|k_c|$ and  
\[
\lambda_1(k_0) >\frac\Lambda2.
\] 
With that $k_0$, thanks to Theorem \ref{ThmSpectral}, we obtain a finite sequence $(\lambda_n, \phi_n)_{1\leq n\leq \sM}$ such that non-trivial function $\phi_n \in H^4((-1,1))$ is a solution of \eqref{EqMain}-\eqref{NavierBound} as $\lambda=\lambda_n$. That helps us to find  a solution  to the system \eqref{SystemMode} as $\lambda=\lambda_n$.  Hence, we define   
\[
\psi_n= - \frac{\phi_n'}{k_0}\quad\text{and}\quad \pi_n = \frac1k (\lambda_n \psi_n-\mu(k_0^2\psi_n-\psi_n'')),
\]
and obtain that 
\[\begin{split}
e^{\lambda_n(k_0) t}( u_{n,1}, u_{n,2}, q_n)^T(k_0, x)= e^{\lambda_n(k_0) t} 
\begin{pmatrix}
 \sin(k_0x_1)\psi_n(k_0,x_2) \\
  \cos(k_0x_1)\phi_n(k_0,x_2) \\
   \cos(k_0x_1)\pi_n(k_0,x_2)
\end{pmatrix}
\end{split}\]
is a real-valued solution to the linearized equations \eqref{EqLinearized}.  

Next, let us split the sequence of characteristic values
\begin{equation}\label{Ordering}
\Lambda > \lambda_1 > \lambda_2> \dots> \lambda_\sN>\frac{\Lambda}2 > \lambda_{\sN+1}> \dots>\lambda_\sM>\dots.
\end{equation}
We formulate two linear combination of normal mode solutions
\begin{equation}
\begin{pmatrix} u^\sN \\ p^\sN \end{pmatrix}(t,x) 
= \sum_{j=1}^\sN \Csf_j e^{\lambda_j t}\begin{pmatrix} u_j \\ q_j \end{pmatrix} (x), \quad \text{and }
\begin{pmatrix} \widetilde u^\sN \\ \widetilde p^\sN \end{pmatrix}(t,x) 
= \sum_{j=1}^{\sN-1} \Csf_j e^{\lambda_j t}\begin{pmatrix} u_j \\ q_j \end{pmatrix} (x)
\end{equation}
Let $\delta >0$, using $\delta v^\sN(0,x)$ as the initial datum,  the nonlinear equations \eqref{EqMain}-\eqref{NavierBound}  admits a local solution $(v^\delta, r^\delta) \in C([0,T^{\max}), H^2(\Omega)\times H^1(\Omega))$, with $(v,r)=(u, p)$ or $(\widetilde u,\widetilde p)$.

Let $0<\varepsilon_0\ll 1$ be fixed later and
\[
F_\sN(t)= \sum_{j=1}^\sN |\Csf_j| e^{\lambda_j t}.
\]
 Hence, there is a unique $T^\delta$ such that $\delta F_\sN(T^\delta)=\varepsilon_0$. Let $C_1 = \|u^\sN(0)\|_{H^2}$,  we define
\begin{equation}\label{T-star}
T^\star = \sup \{ t \in (0,T^{\max}) \mid \|u^\delta(t)\|_{H^2}+\|\tilde u^\delta(t)\|_{H^2} \le 2C_1\delta_0 \},
\end{equation}
and
\begin{equation}\label{T-star2}
T^{\star\star} = \sup \{ t \in (0, T^{\max}) \mid \|u^\delta(t)\|_{L^2}+ \|\tilde u^\delta(t)\|_{L^2} \le 3C_1 \delta F_\sN(t) \}.
\end{equation}
Note that 
\[
\|u^\delta(0)\|_{H^2}+\|\tilde u^\delta(0)\|_{H^2} \leq 2 C_1 \delta < 2C_1\delta_0,
\]
thus $T^\star>0$ is well-defined. Similarly, we have $T^{\star\star}>0$. 

For any $t \le \min\{T^\star, T^{\star\star}, T^\delta\}$, it follows from \cite[Proposition 4.1]{DLX18}  that (for $v=u$ or $\widetilde u$)
\begin{equation}\label{H2bound-V}
\begin{split}
\|v^\delta(t)\|_{H^2} + \|\d_t v^\delta(t)\|_{L^2} 
&\le C \|v^\delta(0)\|_{H^2} + C \int_0^t \|v^\delta(s)\|_{L^2} \, ds\\
&\le C \|v^\delta(0)\|_{H^2} + C\delta \int_0^t F_\sN(s)  \, ds\\
&\le C \delta + C  \delta \sum_{j=1}^\sN \frac{|\Csf_j|}{\lambda_j} (e^{\lambda_j t} -1)\\
&\leq C_2\delta F_\sN(t).
\end{split}
\end{equation}

\subsection{The difference function}
Still let  $(v,r)=(u, p)$ or $(\widetilde u,\widetilde p)$. Denote 
\[
v^d =v^\delta - \delta v^\sN, \quad r^d= r^\delta-\delta r^\sN,
\]
that  solve the boundary value problem
\begin{equation}\label{EqDiff}
\begin{cases}
\d_t v^d + \nabla r^d - \mu \Delta v^d = -v^\delta \cdot \nabla v^\delta,&\quad\text{in }\Omega,\\
\text{div}v^d=0 &\quad\text{in }\Omega,\\
v_2^d =0 &\quad\text{on } \Sigma_+ \cup \Sigma_-,\\
\mu \d_2 v_1^d= \xi_+ v_1^d &\quad\text{on }\Sigma_+,\\
\mu\d_2 v_1^d=- \xi_- v_1^d &\quad\text{on }\Sigma_-,
\end{cases}
\end{equation}
with the initial datum $v^d(0)=0$. 

To obtain the energy estimates for $v^d$, we recall the following property of $\Lambda$, that was proven in \cite[Proposition 4.2]{DLX18}.
\begin{proposition}\label{PropSharpGrowthRate}
Let $\mathbf{w}=(w_1,w_2) \in H_\sigma^1(\Omega) \cap H^2(\Omega)$, then it holds that
\begin{equation}
-\mu \int_{\Omega} |\nabla \mathbf{w}|^2 \, dx+ \xi_+ \int_{2\pi L\T} |w_1(x_1,1)|^2 \, dx_1 + \xi_- \int_{2\pi L\T} |w_1(x_1,-1)|^2 \, dx_1 \le \Lambda \int_{\Omega} |\mathbf{w}|^2 \, dx,
\label{eq:4.27}
\end{equation}
where $\Lambda$ is defined in \eqref{GrowthRate}.
\end{proposition}

Now, multiplying \eqref{EqDiff}$_1$ by $v^d$, we obtain
\begin{equation*}\begin{split}
\frac12 \frac{d}{dt} \int_\Omega |v^d|^2 & = \int_{2\pi L\T} \big( \xi_+|v_1^1(x_1,1)|^2 + \xi_-  |u_1^1(x,-1)|^2 \big) dx_1     -\mu \int_\Omega |\nabla v^d|^2  - \int_\Omega ( v^\delta \cdot \nabla v^\delta) \cdot v^d.
\end{split}\end{equation*}
Notice that
\begin{equation*}
\int v^\delta \cdot \nabla v^\delta \cdot v^d  \le \|u^\delta \cdot \nabla v^\delta\|_{L^2} \|v^d\|_{L^2} \le C\|v^\delta\|_{H^2}^2 \|v^d\|_{L^2},
\end{equation*}
and that (see Proposition \ref{PropSharpGrowthRate})
\begin{equation*}
\int_{2\pi L\T} \big( \xi_+|v_1^1(x_1,1)|^2 + \xi_-  |v_1^1(x,-1)|^2 \big) dx_1   -\mu \int_\Omega |\nabla v^d|^2   \le \Lambda \int |v^d|^2,
\end{equation*}
Combining the previous inequalities  gives us that
\begin{equation*}
\frac{d}{dt} \|v^d\|_{L^2}  \le \Lambda \|v^d\|_{L^2} + C  \|v^\delta\|_{H^2}^2.
\end{equation*}
Making use of the preceding inequality with \eqref{H2bound-V}, we obtain 
\begin{equation*}
\frac{d}{dt} \|v^d\|_{L^2}  \le \Lambda \|v^d\|_{L^2} + C \delta^2 F_\sN^2(t).
\end{equation*}
Thus, using the condition $2\lambda_j - \Lambda > 0$,  it follows from the Gronwall inequality that
\begin{equation}\label{Bound-Ud}
\begin{split}
\|u^d(t)\|_{L^2}  &\le C \delta^2 e^{\Lambda t} \int_0^t e^{-\Lambda s}F_\sN(2s) ds
\notag\\
&\le  C \sum_{j=1}^\sN \delta^2  e^{\Lambda t} \int_0^t e^{(2\lambda_j- \Lambda)s} ds \\
&\leq C_3 \delta^2 F_\sN^2(t).
\end{split}
\end{equation}

\subsection{Proof of Theorem \ref{ThmNonlinear}}

We are in position to prove the main theorem. Note that (due to \eqref{Ordering})
\[
\varepsilon_0 = \delta \sum_{j=1}^\sN |\Csf_j| e^{\lambda_j T^\delta} \leq C_4 \delta  |\Csf_\sN| e^{\lambda_\sN T^\delta}.
\] 
Thus, we   show that
\begin{equation}  
T^\delta = \min\{ T^\star, T^{\star\star}, T^\delta\}, \quad\text{provided that  }\varepsilon_0< \min\big\{ \frac{C_1\delta_0}{2C_2}, \frac{C_1}{C_3}, \frac1{4C_3C_4} \big\}.
\end{equation}
In fact, if $T^\star< T^\delta$, we have (due to \eqref{H2bound-V})
\[
\|(u^\delta,\widetilde u^\delta)(T^\delta)\|_{H^2}\leq 2C_3 \delta F_\sN(T^\delta)  =2C_2 \varepsilon_0 < C_1 \delta_0.
\]
If $T^{\star\star}<T^\delta$, we have (due to \eqref{Bound-Ud})
\[
\begin{split}
\|(u^\delta,\widetilde u^\delta)(T^\delta)\|_{L^2} &\leq \delta \| (u^\sN, \widetilde u^\sN)(T^\delta)\|_{L^2} + \| (u^d, \widetilde u^d)(T^\delta)\|_{L^2} \\
&\leq 2C_1 \delta F_\sN(T^\delta)+ C_3 \delta^2 F_\sN^2(T^\delta) \\
&<3C_1 \delta F_\sN(T^\delta).
\end{split}
\]
Those inequalities contradict to the definition of $T^\star$ \eqref{T-star} and of $T^{\star\star}$ \eqref{T-star2}. That implies $T^\delta = \min\{ T^\star, T^{\star\star}, T^\delta\}$.

 Once we have that $T^\delta \leq \min\{T^\star, T^{\star\star}\}$, we deduce 
\begin{equation}
\begin{split}
\| u^\delta(T^\delta)- \widetilde u^\delta(T^\delta)\|_{L^2} &\geq \delta \|u^\sN(T^\delta)- \widetilde u^\sN(T^\delta)\|_{L^2} - \|u^d(T^\delta)\|_{L^2} -\|\widetilde u^d(T^\delta)\|_{L^2}  \\ 
&\geq  \delta |\Csf_\sN| e^{\lambda_\sN T^\delta} -2 C_3 \delta^2 F_\sN^2(T^\delta) \\
&\geq \frac1{C_4} \varepsilon_0 -2 C_3 \varepsilon_0^2 \geq \frac1{2C_4} \varepsilon_0. 
\end{split}
\end{equation}
Theorem \ref{ThmNonlinear} thus follows.

\section*{Acknowledgements}
 This research was funded by the VNU University of Science, grant number TN.25.21


\begin{thebibliography} {999999}


\bibitem{APV98}
\textsc{Y. Achdou, O. Pironneau, F. Valentin}, Effective boundary conditions for laminar flow over periodic rough boundaries. \textit{J. Comput. Phys.} \textbf{147}, 187--218 (1998).


\bibitem{A++21}
\textsc{P. Acevedo Tapia, C. Amrouche, C. Conca, A. Ghosh}, Stokes and Navier–Stokes equations with Navier boundary conditions, \textit{J. Differ. Equ.}, \textbf{285} (2021), 258--320.  


\bibitem{AR14}
\textsc{C. Amrouche, A. Rejaiba}, $L^p$-theory for Stokes and Navier--Stokes equations with Navier boundary condition, \textit{J. Differ. Eqs.} \textbf{256} (2014), pp. 1515--1547.

\bibitem{AS11}
\textsc{C. Amrouche, N.H. Seloula}, On the Stokes equations with the Navier-type boundary conditions, \textit{Differ. Eqs. Appl.} \textbf{3}(4), pp. 581--607 (2011)


%\bibitem{AdO07}
%\textsc{S. Antontsev, H. de Oliveira},  Navier--Stokes equations with absorption under slip boundary conditions: existence, uniqueness and extinction in time, RIMS K\^okyur\^oku Bessatsu B1, 21--41 (2007).


\bibitem{Ban01}
\textsc{E. B\"ansch}, Finite element discretization of the Navier--Stokes equations with free capillary surface, \textit{Numer. Math.} \textbf{88}, 203--235 (2001)

\bibitem{BJ67}
\textsc{G. Beavers, D. Joseph}, Boundary conditions at a naturally permeable wall, \textit{J. Fluid Mech.} \textbf{30} (1967), pp. 197--207.

\bibitem{Cha61}
\textsc{S. Chandrasekhar}, Hydrodynamics and Hydromagnetic Stability, Oxford University Press, London, 1961.

\bibitem{CS93}
\textsc{D. S. Chauhan, K. S. Shekhawat}, Heat transfer in Couette flow of a compressible Newtonian fluid in the presence of a naturally permeable boundary, \textit{Journal of Physics D: Applied Physics} \textbf{26} (1993), 933--936.

\bibitem{CMR98}
\textsc{T. Clopeau, A. Mikelic, R. Robert}, On the vanishing viscosity limit for the 2D incompressible Navier-Stokes equations with the friction type boundary conditions, \textit{Nonlinearity} (1998) \textbf{11} 1625.

\bibitem{Cor96}
\textsc{J.-M. Coron},   On the controllability of the 2-D incompressible Navier-Stokes equations with the Navier slip boundary conditions, \textit{ESAIM: Control, Optimisation and Calculus of Variations}, Volume 1 (1996), pp. 35-75.

%\bibitem{CS93}
%\textsc{D. S. Chauhan, K. S.  Shekhawat}, Heat transfer in Couette flow of a compressible Newtonian fluid in the presence of a naturally permeable boundary, \textit{J. Phys. D Appl. Phys.} \textbf{26} (1993), pp. 933--936.

%\bibitem{DJL20}
%\textsc{S. Ding, Z. Ji, Q. Li},  Rayleigh--Taylor instability for nonhomogeneous incompressible fluids with Navier-slip boundary conditions, \textit{Math. Meth. Appl. Sci. }(2020), pp. 1--25. 

%\bibitem{DL20}
%\textsc{S. Ding, Z. Lin}, Stability for two-dimensional plane Couette flow to the incompressible Navier--Stokes equations with Navier boundary conditions, \textit{Commun. Math. Sci.} \textbf{18}(5) (2020), pp. 1233--1258.

\bibitem{DLX18}
\textsc{S. Ding, Q. Li, Z. Xin}, Stability analysis for the incompressible Navier–Stokes equations with Navier boundary conditions, \textit{J. Math. Fluid Mech.} \textbf{ 20} (2018), pp. 603--629.

\bibitem{DG03}
\textsc{B. Desjardins, E. Grenier}, Linear instability implies nonlinear instability for various types of viscous boundary layers, \textit{Ann. Inst. H. Poincaré Anal. Non Linéaire} \textbf{20} (2003), pp. 87--106.

\bibitem{DGV25}
\textsc{E. Dormy, D. Gerard-Varet}, Diffusion-free boundary conditions for the Navier-Stokes equations, \textit{preprint} arXiv:2506.17749, 2025.


\bibitem{FBV09}
\textsc{F. Feuillebois, M. Z. Bazant, O. I. Vinogradova} , Effective slip over Superhydrophobic Surfaces in thin channels, \textit{Phys. Rev. Lett.} \textbf{102}, 026001 (2009).

\bibitem{GK12}
\textsc{G. M. Gie,  J. P. Kelliher}, Boundary layer analysis of the Navier–Stokes equations with generalized Navier boundary conditions, \textit{J. Differ. Eqs.} \textbf{253} (2012), pp. 1862--1892.



\bibitem{GS95}
\textsc{Y. Guo, W. Strauss}, Instability of periodic BGK equilibria, \textit{Comm. Pure Appl. Math.} \textbf{ 48} (1995), pp. 861--894.


%\bibitem{Gre00}
%\textsc{E. Grenier}, On the nonlinear instability of Euler and Prandtl equations,  \textit{Comm. Pure Appl. Math.} \textbf{ 53}, (2000), pp.  1067--1091.



%\bibitem{HWLS16}
%\textsc{A. S. Haase, J. A. Wood, R. G. H. Lammertink, J. H. Snoeijer}, Why bumpy is better: the role of the dissipaption distribution in slip flow over a bubble mattress, \textit{Phys. Rev. Fluid} \textbf{}1, 054101 (2016)

\bibitem{JM00}
\textsc{W. J\"ager, A. Mikel\'ic}, On the interface boundary condition of Beavers, Joseph, and Saffman, \textit{SIAM J. Appl. Math.} \textbf{60}, pp. 1111--1127 (2000)

\bibitem{JM01}
\textsc{W. J\"ager, A. Mikel\'ic}, On the roughness-induced effective boundary conditions for an incompressible viscous flow, \textit{J. Differ. Eqs.} \textbf{170}, pp. 96--122 (2001).

\bibitem{Joh02}
\textsc{V. John}, Slip with friction and penetration with resistance boundary conditions for the Navier--Stokes equation-numerical test and aspect of the implementation, \textit{J. Comput. Appl. Math.} \textbf{147}, pp. 287--300 (2002).

\bibitem{IS11}
\textsc{D. Iftimie, F. Sueur}, Viscous boundary layers for the Navier–Stokes equations with the Navier slip conditions, \textit{Arch Rational Mech Anal} \textbf{199}, 145--175 (2011).


\bibitem{HWLS16}
\textsc{A. S. Haase, J. A. Wood, R. G. H. Lammertink, J. H. Snoeijer}, Why bumpy is better: the
role of the dissipaption distribution in slip flow over a bubble mattress, \textit{Phys. Rev. Fluid} \textbf{1} (2016), 054101.


\bibitem{Kel06}
\textsc{J. P. Kelliher}, Navier–Stokes equations with Navier boundary conditions for a bounded domain in plane, \textit{SIAM J. Math. Anal.} \textbf{38}(1) (2006), pp. 210--232.

\bibitem{KT24}
\textsc{J. Koganemaru, I. Tice}, Traveling wave solutions to the free boundary incompressible Navier-Stokes equations with Navier boundary conditions, \textit{J. Differential Equations} \textbf{411} (2024), pp. 381--437.

%\bibitem{Laf01}
%\textsc{O. Lafitte}, Sur la phase lin\'eaire de l'instabilit\'e de Rayleigh-Taylor. S\'eminaire Equations aux D\'eriv\'ees Partielles du Centre de Math\'ematiques de l'Ecole Polytechnique, Ann\'ee 2000–2001.

\bibitem{LN20}
\textsc{O. Lafitte, T.-T. Nguyen}, Spectral analysis of  the incompressible viscous Rayleigh-Taylor system,  \textit{Water Waves} \textbf{4} (2022), pp. 259--305. 


%\bibitem{Lad67}
%\textsc{O. A. Ladyhzenskaya}, Sur de nouvelles \'equations dans la dynamique des fluides visqueux et leurs r\'esolution globale, \textit{Trudy Mat. Inst. Steklov.} \textbf{102} (1967), pp. 85--104.


%\bibitem{Lad66}
%\textsc{O. A. Ladyhzenskaya}, On nonlinear problems of continuum mechanics, \textit{Proc. Int. Congr. Math.  (Moscow, 1966)}, pp. 560--573. Nauka, Moscow, 1968; translation in \textit{Amer. Math. Soc. Transl.} (2) \textbf{70}, 1968.

%\bibitem{Lad68}
%\textsc{O. A. Ladyhzenskaya}, Sur des modifications des \'equations de Navier--Stokes pour des grand gradients de vitesses, \textit{Sem. Math. V. A. Steklov} \textbf{7} (1968), pp. 126--154.

%\bibitem{Lad69}
%\textsc{O. A. Ladyhzenskaya}, \textit{The mathematical theory of viscous incompressible flow}, 2nd ed. Mathematics and Its Applications, 2. Gordon and Breach, New York–London–Paris, 1969.


%\bibitem{LZ17}
%\textsc{H.-L. Li, X. Zhang}, Stability of plane Couette flow for the compressible Navier-Stokes equations with Navier--slip boundary, \textit{J. Differ. Equ.} \textbf{263}(2) (2017), pp. 1160--1187
%\bibitem{Lin98}
%\textsc{J. D. Lindl}, Inertial Confinement Fusion, Springer, 1998. 

%\bibitem{Lio65}
%\textsc{J.-L. Lions},  Sur certaines équations paraboliques non lin\'eaires, \textit{Bull. Soc. Math. France} \textbf{93}  (1965), pp. 155--175.

%\bibitem{Lio69}
%\textsc{J.-L. Lions},  Quelques m\'ethodes de r\'esolution des problèmes aux limites non lin\'eaires. Dunod; Gauthier-Villars, Paris, 1969.

\bibitem{LPZ22}
\textsc{F. Li, R. Pan, Z. Zhang}, Stability and instability of the 3D incompressible viscous flow in a bounded domain, \textit{Calc. Var. Partial Differential Equations}, \textbf{61}(3):Paper No. 95, 26, 2022.

\bibitem{LK24}
\textsc{Y. Lin, R. Kerswell}, Weakening the effect of boundaries: ‘diffusion--free’ boundary conditions as a ‘do least harm’ alternative to Neumann, \textit{Geophysical $\&$ Astrophysical Fluid Dynamics}, 1--25. 
%https://doi.org/10.1080/03091929.2024.2403759

%\bibitem{MRF95}
%\textsc{J. Magnaudet, M. Riverot, J. Fabre}, Accelerated flows past a rigid sphere or a spherical bubble. Part 1. Steady straining flow, \textit{J. Fluid Mech.}  \textbf{284} (1995), pp.  97--135.

\bibitem{Nav23}
\textsc{C. L. Navier}, Sur les lois de l'\'equilibre et du mouvement des corps \'elastiques, \textit{Mem. Acad. R. Sci. Inst. France} (1827), pp. 63--69.

\bibitem{Tai22}
\textsc{T.-T. Nguyen}, Linear and nonlinear analysis of the viscous Rayleigh-Taylor system with Navier-slip boundary conditions, \textit{Calc. Var. and PDEs}, \textbf{63}: 41 (2024).

\bibitem{QWS03}
\textsc{T. Qian, X. Wang, P. Sheng}, Molecular scale contact line hydrodynamics of immiscible flows, \textit{Phys. Rev. E} \textbf{68}, 016306 (2003).


\bibitem{Rot10}
\textsc{J. P. Rothstein}, Slip on superhydrophobic surfaces, \textit{Annual Review of Fluid Mechanics},
\textbf{42} (2010), 89--109.

\bibitem{SS73}
\textsc{V. Solonnikov, V. Scadilov}, A certain boundary value problem for the stationary system of Navier-Stokes equations. \textit{Trudy Mat. Inst. Steklov.} \textbf{125}, pp.196--210 (1973); translation in \textsc{Proc. Steklov Inst. Math.} \textbf{125} (1973), pp. 186--199.

\bibitem{Saf71}
\textsc{P. G. Saffman}, On the boundary condition at the surface of a porous medium,
\textit{Studies in Applied Mathematics} \textbf{50} (1971), 93--101.

\bibitem{Ser59}
\textsc{J. Serrin}, \textit{Mathematical principles of classical fluid mechanics}. In: Truesdell, C. (ed.) Fluid Dynamics I. Encyclopedia of Physics, pp. 125--263. Springer, Berlin (1959).

\bibitem{Tem84}
\textsc{R. Temam}, \textit{Navier--Stokes Equations}, Studies in Mathematics and its Applications 2. North-Holland, Amsterdam (1984)

\bibitem{dV05}
\textsc{H. Beir\~ao da Veiga}, On the regularity of flows with Ladyzhenskaya shear-dependent viscosity and slip or nonslip boundary conditions, \textit{Commun. Pure Appl. Math.} \textbf{LVIII} (2005), pp. 552--577.

\end{thebibliography}
\end{document}